\documentclass[12pt,reqno]{amsart}

\usepackage[mathlines]{lineno}

\usepackage{times}
\usepackage{amsfonts}
\usepackage{amssymb}
\usepackage{latexsym}
\usepackage{xcolor}

\newtheorem{theorem}{Theorem}
\newtheorem{lemma}[theorem]{Lemma}
\newtheorem{proposition}[theorem]{Proposition}

\newtheorem{remark}{Remark}
\newtheorem{corollary}[theorem]{Corollary}
\newtheorem{example}{Example}

\usepackage{comment}

\begin{document}

\begin{center}
\Large{Hausdorff Operators on de Branges Spaces and Paley-Wiener spaces}
\end{center}

\

\centerline{A. R. Mirotin}

\

\centerline{amirotin@yandex.ru}

\

Abstract. For a class of de Branges spaces containing polynomials, sufficient and necessary conditions  are given for the boundedness and compactness of the Hausdorff operators under consideration.  For the Paley-Wiener spaces we reduce the study of  our Hausdorff operators to  classical integral ones. The operators that appeared  are Carleman and therefore closeble in $L^2(\mathbb{R})$. We obtain also conditions for boundedness, compactness and nuclearity of our operators in the   Paley-Wiener space as well as the conditions for their belonging to the Hilbert-Schmidt class.  

\



\section{Introduction}

The idea of the Hausdorff operator goes back to the works of Rogozinsky, Garabedian, Hardy and Littlewood. Though at the first time such operators act in spaces of real functions, later  they were considered in the complex setting, as well (see, e.g., \cite{HKQ, S, KGM, KM, MCAOT, MCAOT2, LM:survay} and the bibliography therein). Last time several papers appeared on Hausdorff operators in spaces of entire functions, e.g., \cite{Blasco, Bonet,  SG}\footnote{For the general concept of a Hausdorff operator see \cite{arxGeneral}.}. For the general state of the art see \cite{LM:survay}. The present  work is devoted to the study of Hausdorff operators on  de Branges spaces and  their special case Paley-Wiener spaces. To our knowledge, these classes of spaces of entire functions have not appeared in the context of Hausdorff operators. 
  
For classes of de Branges spaces containing polynomials, sufficient and necessary conditions  are given for the boundedness and  compactness  of the Hausdorff operators under consideration. To this end we borrow some ideas due to Bonet, Blasco and Galbis \cite{Bonet, Blasco} and use  results obtained by Baranov, Belov, and Borichev  \cite{BBB}.

For the  Paley-Wiener spaces we reduce the study of  Hausdorff operators  $\mathbb{H}_{\Phi,\mu}$ with an arbitrary measure $\mu$ to  classical integral operators on the real line. We show that the operators that appeared  (with an appropriate domain) are Carleman and therefore closeble in $L^2(\mathbb{R})$. This representation gives us also an opportunity to  obtain  conditions for boundedness, compactness and nuclearity of our operators in the   Paley-Wiener space as well as the conditions for their belonging to the Hilbert-Schmidt class.

 \section{Preliminaries} 

In this work we define a Hausdorff operator as
\begin{align}\label{hausdorff}
(\mathbb{H}_{\Phi,\mu}F)(z)=\int_{0}^\infty\Phi(u)F\left(\frac zu\right)d\mu(u), \ z\in \mathbb{C}
\end{align}
where $\mu$ is a regular positive Borel measure on $(0,\infty)$ and the kernel $\Phi$ is a $\mu$-measurable complex-valued function. This class of operators is  connected with operators with homogeneous kernels, see
the book by Karapetiants and Samko \cite{KS}. Such operators with $\Phi(u)=\frac 1u$  were considered  recently  in several spaces of entire functions   \cite{Blasco, Bonet,  SG}. 

\begin{example} If  $\mu$ is a counting  measure and supported in the countable set 
$\{\frac{1}{a_k}:  a_k>0, k\in \mathbb{Z}_+\}$ Hausdorff operator \eqref{hausdorff} turns into a {\it discrete Hausdorff operator} of the form
\begin{align*}
(\mathbb{H}_{c,a}F)(z):=\sum_{k=0}^\infty c_kF(a_k z)
\end{align*}
where $c_k\in  \mathbb{C}$.
\end{example}

Let us recall the definition of a  de Branges space that were introduced and studied  by de Branges \cite{dB} and many of his followers (see, e.~g., \cite{{Dym McKean}, Baranov, BBB, BaranovBo} and the bibliography therein).

An entire function $E$ is said to be in the Hermite-Biehler class if it satisfies
 $|E(z)|>|E^{\#}(z)|$
 for $z$ in the  upper half-plane $\mathbb{C}^+$, where $E^{\#}(z):=\overline{E(\overline{z})}$.
Any such function $E$ determines a de Branges space 
$$
\mathcal{H}(E) =\left\{F \mbox{  entire }:  \frac{F}{E},  \frac{F^{\#}}{E} \in H^2(\mathbb{C}^+)\right\}
$$
($H^2(\mathbb{C}^+)$
 being the standard Hardy space in the upper half-plane) which is a Hilbert
space when equipped with the norm
\begin{align}\label{norm30}
\|F\|_E^2:=\int_{-\infty}^{\infty}\left|\frac{F(t)}{E(t)}\right|^2dt.
\end{align}

An entire function $E$ is said to be of Polya class if it has no zeros in the 
upper half-plane, if $|E(x - \imath y)| < |E(x + \imath y)$ for $y>0$, and if $|E(x + \imath y)|$
is a non-decreasing function of $y> 0$ for each fixed $x$. It is known that each function of Polya class is in the Hermite-Biehler class \cite{dB}.

 \section{The general case of de Branges spaces}

\subsection{Action between two arbitrary spaces}\label{E1E2}

Let us consider such  functions $E_2$ and  $E_1$ of Hermite-Biehler class and positive measures $\mu$  on $(0,\infty)$ that 
\begin{align}\label{condition}
|E_{2}(u(x+iy))|\ge v(u)^{-1}|E_{1}(x+iy)| \mbox{  for  } x\in  \mathbb{R}, y\ge 0 \end{align}
  and for a.e. $ u\in {\rm supp}(\mu)$,
 where $v$ is some positive $\mu$-measurable function.

\begin{remark}\label{ex:30}
Let $E_{m,a}(z)=z^me^{-\imath az}$ with $m\in \mathbb{Z}_+$, $a>0$ (the case $m=0$ corresponds to the  Paley-Wiener space $PW_a$  (see, e.g., \cite{dB})).  The function $E_{m,a}$ with $m\ge 1$ satisfies the condition \eqref{condition}  if and only if $u\in [1,\infty)$. So, here we must consider such $\mu$ that  ${\rm supp}(\mu)\subseteq [1,\infty)$. In this case, the maximal value of $v(u)^{-1}$ equals to $u^m$. In the case of the  Paley-Wiener space $PW_a$ (i. e. $m=0$) one can take $v(u)\equiv 1$ in \eqref{condition}

\end{remark}

We begin with the following result.

\begin{theorem}\label{th:31} Let  the condition \eqref{condition} holds for  functions $E_2$ and  $E_1$ of Hermite-Biehler class, let $\Phi(u)=0$ for $\mu$-a.e.  $ u\in (0,\delta)$ ($\delta>0$), and  $\Phi(u)v(u)\sqrt{u}$ be  $\mu$-integrable on $[\delta,\infty)$.
Then $\mathbb{H}_{\Phi,\mu}$ is bounded as an operator between  $\mathcal{H}(E_1)$ and 
 $\mathcal{H}(E_2)$ and
\begin{align}\label{est:norm}
\|\mathbb{H}_{\Phi,\mu}\|_{\mathcal{H}(E_1)\to \mathcal{H}(E_2)}\le\int_{\delta}^\infty|\Phi(u)|v(u)\sqrt{u} d\mu(u).
\end{align}
\end{theorem}

\begin{proof} Since  $\delta>0$ and $\Phi\in L^1((\delta,\infty),\mu)$, the function
\begin{equation*}\label{hausdorff_delta_dB} 
z\mapsto \int_{\delta}^\infty \left|\Phi(u)F\left(\frac z u\right)\right|\,d\mu(u)
\end{equation*}
is locally bounded on ${\mathbb C}$ for entire $F$. Then by the Theorem in \cite{Mattner} on the differentiability of integrals depending on a complex parameter,  the function $\mathbb{H}_{\Phi,\mu}F$ is  entire, as well.

Using the Minkowski inequality and \eqref{condition} we have for $F\in \mathcal{H}(E_1)$  and $y\ge 0$
\begin{align*}
&\left\|\frac{\mathbb{H}_{\Phi,\mu}F(\cdot+iy)}{E_2(\cdot+iy)}\right\|_{L^2(\mathbb{R})}\\
&=\left(\int_{-\infty}^\infty\left|\int_{0}^\infty\Phi(u)F\left(\frac{x+iy}{u}\right)d\mu(u)\right|^2\frac{dx}{|E_2(x+iy)|^2}\right)^{\frac 12}\\
&\le 
\int_{0}^\infty\left(\int_{-\infty}^\infty |\Phi(u)|^2\left|F\left(\frac{x+iy}{u}\right)\right|^2\frac{dx}{|E_2(x+iy)|^2}\right)^{\frac 12}d\mu(u)\\
&=\int_{0}^\infty |\Phi(u)|\left(\int_{-\infty}^\infty\left|F\left(\frac{x+iy}{u}\right)\right|^2\frac{dx}{|E_2(x+iy)|^2}\right)^{\frac 12}d\mu(u)\\
&=\int_{0}^\infty |\Phi(u)|\left(\int_{-\infty}^\infty|F(t+\frac{iy}{u})|^2u\frac{dt}{|E_2(u(t+\frac{iy}{u}))|^2}\right)^{\frac 12}d\mu(u)\\
&\le \int_{0}^\infty |\Phi(u)|\left(\int_{-\infty}^\infty|F(t+\frac{iy}{u})|^2uv(u)^2
\frac{dt}{|E_1(t+\frac{iy}{u})|^2}\right)^{\frac 12}d\mu(u)\\
&=\int_{0}^\infty|\Phi(u)|v(u)\sqrt{u}\left(\int_{-\infty}^\infty\left|\frac{F(t+\frac{iy}{u})}{E_1(t+\frac{iy}{u})}\right|^2\right)^{\frac 12}d\mu(u)\\
&\le\int_{0}^\infty|\Phi(u)||v(u)\sqrt{u} d\mu(u)\left\|\frac{F}{E_1}\right\|_{H^2(\mathbb{C}^+)}.
\end{align*}

It follows that $\mathbb{H}_{\Phi,\mu}F\in 
\mathcal{H}(E_2)$, since
\begin{align*}
\left\|\frac{\mathbb{H}_{\Phi,\mu}F}{E_2}\right\|_{H^2(\mathbb{C}^+)}\le
\int_{0}^\infty|\Phi(u)||v(u)\sqrt{u} d\mu(u)\left\|\frac{F}{E_1}\right\|_{H^2(\mathbb{C}^+)}<\infty,
\end{align*}
and
\begin{align*}
\|\mathbb{H}_{\Phi,\mu}F\|_{\mathcal{H}(E_2)}\le
\int_{0}^\infty|\Phi(u)|v(u)\sqrt{u} d\mu(u)\|F\|_{\mathcal{H}(E_1)}.
\end{align*}

\end{proof}

\begin{remark}
Putting $z=0$ in \eqref{hausdorff} we get that the condition $\Phi(u)\in L^1(\mu)$ is necessary for the boundedness of  $\mathbb{H}_{\Phi,\mu}$ as an operator between  $\mathcal{H}(E_1)$ and  $\mathcal{H}(E_2)$ if $\mathcal{H}(E_1)$ contains a function $F$ with $F(0)\ne 0$. 
 \end{remark}
\begin{remark} Below we shall show that the condition  $\Phi(u)=0$ for $\mu$-a.e.  $ u\in (0,\delta)$ for some $\delta>0$ in the previous theorem can not be omitted, in general.
\end{remark}
 
 For  functions of Polya class we have the following corollary.

\begin{corollary}\label{th:Eamp} 
In order that  the Hausdorff operator  \eqref{hausdorff} is bounded in $\mathcal{H}(E_{m,a})$ for some $m\ge 0$ it is sufficient that $\Phi(u)=0$ for $\mu$-a.e. $u\in (0,\delta)$ and $\Phi \in L^1([\delta,\infty)\mu)$. In this case
$$
\|\mathbb{H}_{\Phi,\mu}\|\le\int_{\delta}^\infty|\Phi(u)|u^{-m+\frac 12} d\mu(u).
$$
\end{corollary}

\begin{proof}  The sufficiency follows from Theorem \ref{th:31}, since in the case $E_1=E_2=E_{m,a}$ the condition \eqref{condition} holds for $v(u)^{-1}= u^m$.
\end{proof}

\subsection{The case of spaces with polynomials}\label{monom}

In this section we consider de Branges spaces which contain the set $\mathcal{P}$ of all  polynomials (see, e.g., \cite{Baranov} for the theory and examples of  such spaces). For these spaces one can obtain  necessary conditions and (under additional constrains) some criteria for boundedness of a Hausdorff operator. Since we want  $\mathbb{H}_{\Phi,\mu}$  to  be defined on all monomials  $q_n(z)=z^n$  for all $n\in \mathbb{Z}_+$, we need to assume that the integrals 
\begin{equation}\label{lambda_n}
 \lambda_n:=\int_0^\infty\Phi(u)u^{-n}\,d\mu(u), \ \  n\in \mathbb{Z}_+
 \end{equation}
 exist.
 
As mentioned in the introduction, in this subsection we borrow some ideas from \cite{Bonet, Blasco} and use results obtained in \cite{BBB}.
 
 The following theorem gives necessary conditions for the boundedness of  $\mathbb{H}_{\Phi,\mu}$  in $\mathcal{H}(E)$.
 
 \begin{theorem}\label{th:necessary} Let $\mathcal{P}\subset \mathcal{H}(E)$ and $\mathbb{H}_{\Phi,\mu}$ be bounded  in $\mathcal{H}(E)$. Then

 (i)   $\sup\limits_{n\in  \mathbb{Z}_+}|\lambda_n|<r(\mathbb{H}_{\Phi,\mu})$, the spectral radius of the operator $\mathbb{H}_{\Phi,\mu}$.
 
 (ii)  If, in addition, $\Phi\ge 0$, and $\mathbb{H}_{\Phi,\mu}\ne 0$ then $\liminf\limits_{n\to\infty}\sqrt[n]{\lambda_n}>0$ and $\Phi(u)=0$ for $\mu$-a.e. $u\in (0,\frac 1A)$ where  $A:=\sup\limits_{n\in  \mathbb{Z}_+}\sqrt[n]{\lambda_n}<\infty$.
  \end{theorem}
  
  \begin{proof} (i)  Evidently $q_n\in \mathcal{H}(E)$
   and 
  $$
 \|q_n\|_E^2=\int_{-\infty}^ \infty\left|\frac{t^n}{E(t)}\right|^2\,dt>0
  $$
 for all $ n\in \mathbb{Z}_+$. Moreover, the monomials  $q_n$ are eigenfunctions of $\mathbb{H}_{\Phi,\mu}$ with eigenvalues  $\lambda_n$. Thereby $| \lambda_n|\le r(\mathbb{H}_{\Phi,\mu})$  ($ n\in \mathbb{Z}_+$).
 
 (ii) (Cf. \cite[Lemma 4.1]{Blasco}.) 
  For an arbitrary $0<b<\infty$ and $ n\in \mathbb{Z}_+$ we have 
 \begin{equation}\label{below}
 \lambda_n\ge \int_0^b\Phi(u)u^{-n}\,d\mu(u)\ge \frac{1}{b^n}\int_0^b\Phi(u)\,d\mu(u),
  \end{equation}
 and therefore
 \begin{equation}\label{sqrt}
 \sqrt[n]{\lambda_n}\ge  \frac{1}{b}\sqrt[n]{\int_0^b\Phi(u)\,d\mu(u)}.
 \end{equation}
Since $\mathbb{H}_{\Phi,\mu}\ne 0$ and \eqref{below} holds, the  inequality $0<\int_0^b\Phi(u)\,d\mu(u)<\infty$ is valid for some $b\in (0,\infty)$. Thus, the 
 inequality \eqref{sqrt} implies  $\liminf\limits_{n\to\infty}\sqrt[n]{\lambda_n}\ge \frac 1b$.
 
 Finally, the property (i) shows that $A<\infty$. Then, for each  $\delta<\frac 1A$ one has for all $n$,
 \begin{align*}
 A^n\ge \lambda_n\ge  \int_0^\delta\Phi(u)u^{-n}\,d\mu(u)\ge \frac{1}{\delta^n}\int_0^\delta\Phi(u)\,d\mu(u).
 \end{align*}
 It follows that
 $$
 \int_0^\delta\Phi(u)\,d\mu(u)\le (A\delta)^n\to 0 \mbox{ as } n\to\infty,
 $$
 and therefore $\Phi=0$ $\mu$-a.e. on $(0,\delta)$.
  \end{proof}

\begin{corollary}\label{cor:comp_9} Let $\mathcal{P}\subset \mathcal{H}(E)$. If   $\mathbb{H}_{\Phi,\mu}$ is  a compact operator  in $\mathcal{H}(E)$ then $\lim\limits_{n\to\infty}\lambda_n=0$.
\end{corollary}

\begin{corollary} Under the assumptions of Theorem \ref{th:necessary} (ii) one has 
\begin{align*}\label{hausdorff_delta}
(\mathbb{H}_{\Phi,\mu}F)(z)=\int_{\frac 1A}^\infty\Phi(u)F\left(\frac zu\right)d\mu(u).
\end{align*}

\end{corollary}

The following two theorems give us (under some additional assumptions) a criteria for the boundedness of our Hausdorff operator in de Branges spaces.

First we recall (see, e.g., \cite{BBB}), that to every  de Branges space there correspond two sequences of reals $(t_n)$ and $(\nu_n)$  the so called {\it spectral data} for the space.

The sequence $(t_n)$ is called {\it lacunary} if for some $\kappa>0$ we have $t_{n+1}-t_n \ge \kappa |t_n|$.

We call the spectral data   $((t_n),(\nu_n))$ for the de Branges  space  {\it regular} if for some $c > 0$ and any $n$,
$$
\sum\limits_{|t_k|\le |t_n|}\nu_k+t^2_n\sum\limits_{|t_k|>|t_n|}\frac{\nu_k}{t_k^2}\le c\nu_n.
$$

We need also the following notion. Let
$\varphi: [0,\infty) \to(0,\infty)$ be a measurable function. With each $\varphi$ we associate a {\it radial Fock-type
space} (or a Bargmann--Fock space)
$$
\mathcal{F}_{\varphi}=\{F \mbox { entire}: \|F\|^2_{\mathcal{F}_{\varphi}}=\int_{\mathbb C}|F(z)|^2e^{-\varphi(|z|)}\,dm(z)<\infty\}.
$$
Here $m$ stands for the area Lebesgue measure. 

Note that 
\begin{equation}\label{norm_Fock}
 \|F\|^2_{\mathcal{F}_{\varphi}}=2\pi\int_0^\infty M_2(F,r)^2re^{-\varphi(r)}dr,
\end{equation}
where
\begin{align}\label{norm_Fock2}
M_2(F,r)^2&=\frac{1}{2\pi}\int_0^{2\pi}|F(re^{\imath\theta})|^2\,d\theta\\ \nonumber
&=\sum_{n=0}^\infty |a_n|^2r^{2n},
\end{align}
if   $F$ has the Taylor expansion 
\begin{equation}\label{Taylor_dB}
F(z)=\sum_{n=0}^\infty a_n z^n,\quad z\in {\mathbb C}.
\end{equation}

\begin{theorem}\label{th:necessary_sufficient_dB} Let $\mathcal{P}\subset \mathcal{H}(E)$.  Let the   spectral data $((t_n),(\nu_n))$ for a de Branges space $\mathcal{H}(E)$   be  regular and $(t_n)$ be lacunary. Assume also that $\Phi(u)=0$ $\mu$-a.e. on some interval $(0, \delta)$.  Then  $\mathbb{H}_{\Phi,\mu}$ is bounded in $\mathcal{H}(E)$ if and only if integrals  \eqref{lambda_n} exist and $C:=\sup\limits_{n\in  \mathbb{Z}_+}|\lambda_n|<\infty$.
In this case,
\begin{equation}\label{eq:bdd_dB1}
\|\mathbb{H}_{\Phi,\mu}\|_{\mathcal{H}(E)\to\mathcal{H}(E)}\le {\rm const}\sup\limits_{n\in  \mathbb{Z}_+}|\lambda_n| 
\end{equation}
where the constant  ${\rm const}$  depends on $E$ only.
\end{theorem}

\begin{proof} The necessity. 
 Let  $\mathbb{H}_{\Phi,\mu}$ be bounded in $\mathcal{H}(E)$.  Since $\mathcal{P}\subset \mathcal{H}(E)$,  integrals  \eqref{lambda_n} exist. The necessity of $C<\infty$ follows from Theorem \ref{th:necessary}.
 
 The sufficiency. The space $\mathcal{H}(E)$ satisfies all the conditions of \cite[Theorem 1.2]{BBB},  and thus $\mathcal{H}(E)$ coincides with some  generalized Fock space $\mathcal{F}_{\varphi}$ with equivalence of norms by the aforementioned theorem.
It is sufficient to show that $\mathbb{H}_{\Phi,\mu}$ is bounded on  $\mathcal{F}_{\varphi}$. Since  $\delta>0$ and $\Phi(u)\in L^1(\mu)$, the function
\begin{equation*}\label{hausdorff_delta_dB} 
z\mapsto \int_{\delta}^\infty \left|\Phi(u)F\left(\frac z u\right)\right|\,d\mu(u)
\end{equation*}
is locally bounded on ${\mathbb C}$ for entire $F$. Then, as was mentioned in the proof of Corollary 2.4 in \cite{Bonet}, by the Theorem in \cite{Mattner} on the differentiability of integrals depending on a complex parameter,  for each $n \in {\mathbb Z}_+$, we have for all $z$
\begin{align*}
\left(\frac{d^n}{dz^n}\mathbb{H}_{\Phi,\mu}F\right)(z)&=\int_{\delta}^\infty \Phi(u)\frac{d^n}{dz^n}F\left(\frac z u\right)\,d\mu(u)\\
&=\int_{\delta}^\infty \Phi(u)\frac{1}{u^n}F^{(n)}\left(\frac z u\right)\,d\mu(u),
\end{align*}
and thus the function $\mathbb{H}_{\Phi,\mu}F$ is  entire, as well. Further,  if   $F$ has the Taylor expansion \eqref{Taylor_dB}, we have by the previous equality
$$
\frac{1}{n!}\left(\mathbb{H}_{\Phi,\mu}F\right)^{(n)}(0)=\lambda_n\frac{F^{(n)}(0)}{n!}=\lambda_n a_n.
$$
Therefore  for all $z$
$$
(\mathbb{H}_{\Phi,\mu}F)(z)=\sum_{n=0}^\infty\lambda_n a_n z^n.
$$
Taking into account \eqref{norm_Fock} and \eqref{norm_Fock2} one has
\begin{align}\label{est_9}
\|\mathbb{H}_{\Phi,\mu}F\|_{\mathcal{F}_\varphi}^2&=\int_{0}^\infty\sum_{n=0}^\infty|\lambda_n|^2 |a_n|^2 r^{2n}re^{-\varphi(r)}\,dr\\ \nonumber
&\le C^2\int_{0}^\infty\sum_{n=0}^\infty |a_n|^2 r^{2n}re^{-\varphi(r)}\,dr\\ \nonumber
&= C^2 \int_{0}^\infty M_2^2(F,r)re^{-\varphi(r)}\,dr\\ \nonumber
&=C^2\|F\|_{\mathcal{F}_\varphi}^2. \nonumber
\end{align}
Since there  are such $a,b>0$ that 
$$
a\|F\|_{\mathcal{H}(E)}\le \|F\|_{\mathcal{F}_{\varphi}}\le b\|F\|_{\mathcal{H}(E)}
$$
 for all  $F\in \mathcal{H}(E)$, the inequality \eqref{eq:bdd_dB1} holds with ${\rm const}={\frac ba}$. This completes the proof.
\end{proof}

\begin{theorem}\label{th:necessary_sufficient_dB2} Let $\mathcal{P}\subset \mathcal{H}(E)$. Let the   spectral data $((t_n),(\nu_n))$ for a de Branges space $\mathcal{H}(E)$   be  regular and $(t_n)$ be lacunary. 
Assume also that $\Phi(u)\ge 0$.  Then  $\mathbb{H}_{\Phi,\mu}$ is bounded in $\mathcal{H}(E)$ if and only if integrals  \eqref{lambda_n} exist,   $C=\sup\limits_{n\in  \mathbb{Z}_+}\lambda_n<\infty$,  and $\Phi(u)=0$ $\mu$-a.e. on $(0, 1)$.
In this case,
\begin{equation}\label{eq:bdd_dB}
\|\mathbb{H}_{\Phi,\mu}\|_{\mathcal{H}(E)\to\mathcal{H}(E)}\le {\rm const} \int_{1}^\infty\Phi(u)\,d\mu(u)
\end{equation}
where the constant  ${\rm const}$  depends on $E$ only.
\end{theorem}

\begin{proof} The necessity. 
 Let  $\mathbb{H}_{\Phi,\mu}$ be bounded in $\mathcal{H}(E)$. Since $\mathcal{P}\subset \mathcal{H}(E)$, integrals  \eqref{lambda_n} exist. The necessity of $C<\infty$ follows from Theorem \ref{th:necessary}.  Further, for each $0<\delta<1$ one has
$$
\left(\frac{1}{\delta}\right)^n\int_{0}^\delta\Phi(u)\,d\mu(u)\le \int_{0}^\delta \frac{\Phi(u)}{u^{n}}\,d\mu(u)\le C<\infty,\quad n\in \mathbb{N}.
$$
It follows that $\int_{0}^\delta\Phi(u)\,d\mu(u)=0$ for every $0<\delta<1$, and thus $\Phi(u)=0$ $\mu$-a.e. on $(0, 1)$.

The sufficiency. As in the proof of the previous  theorem  $\mathcal{H}(E)$ coincides with some  generalized Fock space $\mathcal{F}_{\varphi}$ with equivalence of norms.
It is sufficient to show that $\mathbb{H}_{\Phi,\mu}$ is bounded on  $\mathcal{F}_{\varphi}$. Since   $\Phi(u)=0$ $\mu$-a.e.  on $(0,  1)$,

\begin{equation}\label{hausdorff_delta_1} 
(\mathbb{H}_{\Phi,\mu}F)(z)=\int_{1}^\infty\Phi(u)F\left(\frac zu\right)d\mu(u).
\end{equation}
 We have by the Minkowski's inequality
\begin{align*}
M_2(\mathbb{H}_{\Phi,\mu}F,r)&=\left(\frac{1}{2\pi}\int_0^{2\pi}\left|\int_1^\infty\Phi(u)F(\frac{re^{\imath\theta}}{u})d\mu(u)\right|^2\,d\theta\right)^{\frac 12}\\
&\le \frac{1}{\sqrt{2\pi}}\int_1^\infty\left(\int_0^{2\pi}\Phi(u)^2|F(\frac{re^{\imath\theta}}{u})|^2\,d\theta\right)^{\frac 12}\,d\mu(u)\\
&=\frac{1}{\sqrt{2\pi}}\int_1^\infty\Phi(u)M_2(F,\frac ru)\,d\mu(u)\\
&\le \frac{1}{\sqrt{2\pi}}\int_1^\infty\Phi(u)\,d\mu(u)M_2(F,r)
\end{align*}
(above we used Hardy's Theorem).

Taking into account this estimate one has
\begin{align*}
 \|\mathbb{H}_{\Phi,\mu}F\|^2_{\mathcal{F}_{\varphi}}&=2\pi\int_0^\infty M_2(\mathbb{H}_{\Phi,\mu},r)^2re^{-\varphi(r)}dr\\
 &\le \left(\int_1^\infty\Phi(u)\,d\mu(u)\right)^2\int_0^\infty M_2(F,r)^2re^{-\varphi(r)}dr\\
 &=\left(\int_1^\infty\Phi(u)\,d\mu(u)\right)^2\|F\|^2_{\mathcal{F}_{\varphi}}.
\end{align*}
The rest of the proof is the same one as in the proof of the previous  theorem.  
\end{proof}

\begin{theorem}\label{th:compactness_9}
Let the conditions of Theorem \ref{th:necessary_sufficient_dB} hold. Then $\mathbb{H}_{\Phi,\mu}$ is compact in $\mathcal{H}(E)$ if and only if  integrals in \eqref{lambda_n} exist  and $\lim_{n\to  \infty}\lambda_n=0$.
\end{theorem}

\begin{proof} The necessity follows from Corollary \ref{cor:comp_9}  and   Theorem \ref{th:necessary_sufficient_dB}.

The sufficiency. Assume that integrals in \eqref{lambda_n} exist  and $\lim_{n\to  \infty}\lambda_n=0$.  Let a function $F\in \mathcal{H}(E)$ has the Taylor expansion $F(z)=\sum_{n=0}^\infty a_n z^n$ and  $F_k(z):=\sum_{n=0}^{k-1} a_n z^n$. 
Then for each natural  $k$ the operator $\mathbb{H}_{\Phi,\mu}^{(k)}F:=\mathbb{H}_{\Phi,\mu}F_k$ is  finite  dimensional. Moreover, it is  bounded in $\mathcal{H}(E)$ because  $\mathcal{H}(E)$ is isomorphic to $\mathcal{F}_\varphi$ and
$$
\|\mathbb{H}_{\Phi,\mu}^{(k)}F\|_{\mathcal{F}_{\varphi}}\le  \|\mathbb{H}_{\Phi,\mu}\|\|F_k\|_{\mathcal{F}_{\varphi}}\le \|\mathbb{H}_{\Phi,\mu}\|\|F\|_{\mathcal{F}_{\varphi}}
$$
 by the  Theorem \ref{th:necessary_sufficient_dB} and formulas \eqref{norm_Fock}  and \eqref{norm_Fock2}. So, the operator $\mathbb{H}_{\Phi,\mu}^{(k)}$ is compact  in $\mathcal{H}(E)$.
 
 Further, formula \eqref{est_9} shows that 
\begin{align*}
\|\mathbb{H}_{\Phi,\mu}F-\mathbb{H}_{\Phi,\mu}^{(k)}F\|_{\mathcal{F}_\varphi}^2
&=\|\mathbb{H}_{\Phi,\mu}(F-F_k)\|_{\mathcal{F}_\varphi}^2\\
&\le \sup_{n\ge k}|\lambda_n|^2\int_{0}^\infty\sum_{n=k}^{\infty} |a_n|^2 r^{2n}re^{-\varphi(r)}\,dr\\ 
&\le \sup_{n\ge k}|\lambda_n|^2\int_{0}^\infty\sum_{n=0}^\infty |a_n|^2 r^{2n}re^{-\varphi(r)}\,dr\\
&=\sup_{n\ge k}|\lambda_n|^2\|F\|_{\mathcal{F}_\varphi}^2. 
\end{align*}
It follows that $\|\mathbb{H}_{\Phi,\mu}-\mathbb{H}_{\Phi,\mu}^{(k)}\|\to 0$ as $k\to\infty$. This completes the proof.
\end{proof}
 
 \begin{remark} The proofs of theorems  \ref{th:necessary_sufficient_dB},    \ref{th:necessary_sufficient_dB2}, and  \ref{th:compactness_9} are not direct. It may be interesting to obtain  direct proofs of similar (or more general) results. 

\end{remark}

\section{Hausdorff operator in the  Paley-Wiener space}\label{PW}

  Recall that by the Paley-Wiener theorem the   space $PW_{a}=\mathcal{H}(e^{-\imath a z})$ is the Fourier image of the subspace $L^2(-a,a)$ of $L^2(\mathbb{R})$. This theorem asserts that $PW_{a}$ equals to the class of functions which are entire, of exponential type $a$, and whose restrictions to
the real axis belong to $L^2(\mathbb{R})$ (e.g., \cite{dB, Higgins}). Thus, this space does not contain non-trivial polynomials, and the results obtained in the previous subsection  do not apply.

 Without loss of generality we shall consider the  Paley-Wiener space $PW:=PW_{\pi}$.
 
\begin{proposition}\label{th:PW} Assume that $\Phi(u)=0$ $\mu$-a.e. on some interval $(0, \delta)$.  

1) The Hausdorff operator \eqref{hausdorff}
is bounded on the Paley-Wiener space $PW$ if   $\Phi(u)\sqrt{u}$ is $\mu$-integrable on $[\delta,\infty)$ and in this case
$$
\|\mathbb{H}_{\Phi,\mu}\|\le\int_{\delta}^\infty|\Phi(u)|\sqrt{u}\,d\mu(u).
$$

2) If the operator \eqref{hausdorff}  is defined on the Paley-Wiener space $PW$ then   $\Phi$ is $\mu$-integrable on $[\delta,\infty)$. 
\end{proposition}

\begin{proof} 1). This follows from Theorem \ref{th:31}, since for the  Paley-Wiener space one can take $v(u)\equiv 1$.

2)
  It is well known that the function
  $$
{\rm sinc}(z):=\frac{\sin(\pi z)}{\pi z}.
$$
  belongs to $PW$.  Since ${\rm sinc}(0)=1$, the condition $\Phi\in L^1(\mu)$ holds if  the operator \eqref{hausdorff} is defined on some subspace of  $PW$ which contains ${\rm sinc}$.
 \end{proof}

\begin{corollary} Let the measure $\mu$ be supported on some segment  $[\delta,b]$ ($0<\delta<b$). Then the Hausdorff operator \eqref{hausdorff}
is bounded on the Paley-Wiener space $PW$ if and only if $\Phi\in L^1(\mu)$.
\end{corollary}

Now we are aimed to re-right the integral operator \eqref{hausdorff} in a classical form for an arbitrary measure $\mu$ and to show that the operator that appeared  is Carleman.

Recall that an operator $T$ from $L^2(M_1)$ into $L^2(M_2)$ where $M_1$ and $M_2$ are two measure spaces is a  Carleman 
operator if  there exists a measurable function $K : M_2\times  M_1\to \mathbb{C}$ 
such that $K(x, \cdot) \in  L^2(M_1)$ almost everywhere in $M_2$ and 
$(Tf)(x) = \int_{M_1} K(x,y)f(y)\,dy$ almost everywhere in $M_2$, $f\in {\rm dom}(T)$ the domain of $T$.
Such a kernel $K$ is called a Carleman kernel.

 For the theory of Carleman operators we refer to   \cite{Weidman}  (see also \cite{Korotkov, Halmos_S}).

Let
$$
D:=\{ f: f=\widehat\psi, \psi\in C^2( \mathbb{R}), {\rm supp}(\psi)\subset [-\pi,\pi] \}
$$
(the ``hat'' stands for the Fourier transform).

\begin{lemma}\label{lm:dense_dB} $D$ is a dense subspace of $PW$.
\end{lemma}

\begin{proof} The space  $\{\psi\in C^2( \mathbb{R}): {\rm supp}(\psi)\subset [-\pi,\pi] \}$ is dense in $ C^2[-\pi,\pi]$ with respect to the $L^2$ metric.  Since $ C^2[-\pi,\pi]$ is dense in  $L^2(-\pi,\pi)$,   $D$ is a dense subspace of $PW$.
\end{proof}

\begin{theorem}\label{th:K} Let  $\Phi\in L^1(\mu)$.  Then the operator $\mathbb{H}_{\Phi,\mu}$  with the domain $D$ is equal  to a Carleman operator $\mathbb{K}_{\Phi,\mu}$   in $L^2( \mathbb{R})$. More precisely,  for each $f\in D$ and $t\in \mathbb{R}$
 \begin{equation}\label{K}
 (\mathbb{H}_{\Phi,\mu}f)(t)=\int_{-\infty}^{\infty} K_{\Phi,\mu}(t,x)f(x)\,dx,
\end{equation}
where
$$
K_{\Phi,\mu}(t,x)=\int_0^\infty \Phi(u){\rm sinc}\left(\frac t u -x\right)\,d\mu(u)
$$
is a Carleman kernel.
\end{theorem}

\begin{proof} It is known that the function ${\rm sinc}(t-x)$ is a reproducing kernel for $PW$ (see, e.g.,  \cite{Higgins}). In particular, for $f\in D$, one has
$$
f\left({\frac t u}\right)=\int_{-\infty}^{\infty}f(x){\rm sinc}\left({\frac t u}-x\right)\,dx.
$$
Then by Fubini's Theorem
\begin{align*}
(\mathbb{H}_{\Phi,\mu}f)(t)&=\int_{0}^{\infty}\Phi(u)\int_{-\infty}^{\infty}f(x){\rm sinc}\left({\frac t u}-x\right)\,dx\,d\mu(u)\\
&=\int_{-\infty}^{\infty}f(x)\left(\int_0^\infty \Phi(u){\rm sinc}\left(\frac t u -x\right)\,d\mu(u)\right)\,dx\\
&=\int_{-\infty}^{\infty} K_{\Phi,\mu}(t,x)f(x)\,dx.
\end{align*}
The application of Fubini's Theorem is correct. Indeed,
since $f\in L^2(\mathbb{R})$ and $f=\widehat\psi$ where $\psi\in C^2(\mathbb{R})\cap L^1(\mathbb{R})$ for $f\in D$, we have  $f(x)=O(x^{-2})$ as $x\to\infty$, and thus
 $f\in L^1(\mathbb{R})$. It follows that
\begin{align*}
&\int_{-\infty}^{\infty}|f(x)|\int_0^\infty |\Phi(u)| \left|{\rm sinc}\left(\frac t u -x\right)\right|\,d\mu(u)\,dx\\
\le& \int_0^\infty |\Phi(u)|\,d\mu(u)\int_{-\infty}^{\infty}|f(x)|\,dx<\infty.
\end{align*}

Further, we have by the Minkowski inequality
\begin{align}\label{Carleman}
\|K_{\Phi,\mu}(t,\cdot)\|_{L^2(\mathbb{R})}&=\left(\int_{-\infty}^\infty |K_{\Phi,\mu}(t,x)|^2\,dx \right)^{\frac 12}\\  \nonumber
 &=\left(\int_{-\infty}^\infty \left|\int_0^\infty \Phi(u){\rm sinc}\left(\frac t u -x\right)\,d\mu(u)\right|^2\,dx \right)^{\frac 12}\\ \nonumber
 &\le \int_0^\infty \left(\int_{-\infty}^\infty \left|\Phi(u){\rm sinc}\left(\frac t u -x\right)\right|^2\,dx\right)^{\frac 12}\,d\mu(u)\\ \nonumber
 &= \int_0^\infty |\Phi(u)|\left(\int_{-\infty}^\infty {\rm sinc}^2\left(\frac t u -x\right)\,dx\right)^{\frac 12}\,d\mu(u)\\ \nonumber
 &= \int_0^\infty |\Phi(u)|\,d\mu(u)=\|\Phi\|_{L^1(\mu)}.
\end{align}
In particular $K_{\Phi,\mu}$ is a  Carleman kernel.

Denote by $\mathbb{K}_{\Phi,\mu}$ the integral  operator with the kernel $K_{\Phi,\mu}$. This operator  maps $D$ into $ L^2(\mathbb{R})$ because for all $f\in D$ one has by the Minkowski inequality
\begin{align*}
\|\mathbb{K}_{\Phi,\mu}f\|_{L^2(\mathbb{R})}&=\left\|\int_{-\infty}^\infty K_{\Phi,\mu}(\cdot,x)f(x)\,dx \right\|_{L^2(\mathbb{R})}\\  
&\le \int_{-\infty}^\infty \|K_{\Phi,\mu}(\cdot,x)\|_{L^2(\mathbb{R})}|f(x)|\,dx \\
&\le \|\Phi\|_{L^1(\mu)}\|f\|_{L^1(\mathbb{R})}<\infty.
\end{align*}

 Finally, consider the vector-valued function $k:\mathbb{R}\to L^2(\mathbb{R})$, $k(t):=K_{\Phi,\mu}(t,\cdot)$. Then $k$ is an inducing  function of the Carleman operator (below $\langle\cdot,\cdot \rangle$ stands for the inner product in $L^2(\mathbb{R})$)
$$
(Tf)(t)=\langle k(t),f \rangle
$$
with the domain $ {\rm dom}(T):=D^*$ where the star denotes  the complex conjugation (see, e.g., \cite[p. 141]{Weidman}, or  \cite[p. 63]{Halmos_S} where $k$ is called a Carleman function). If we consider the (bounded) operator $S$ in $L^2(\mathbb{R})$ of complex conjugation, $Sf=f^*$, then the  operator $\mathbb{K}_{\Phi,\mu}=TS$ with the domain $D$ is also Carleman (see, e.~g.,  \cite[Theorem 6.13]{Weidman}).
\end{proof}

The previous theorem gives us an opportunity to apply to our case the classical theory of integral operators. For example, the following corollaries are valid (in the following we assume that   $\Phi\in L^1(\mu)$). Recall also that  $D$ is dense in $PW$.

\begin{corollary} The  operator $\mathbb{H}_{\Phi,\mu}$  with the domain $D$ is closable in $L^2(\mathbb{R})$ if  $\Phi\in L^1(\mu)$.
\end{corollary}

 \begin{proof} Indeed, each Carleman operator is closable  (see \cite[Theorem 6.13]{Weidman}).
  \end{proof}
  
 \begin{remark} In fact the Carleman operator $\mathbb{K}_{\Phi,\mu}$ with the domain 
 $$
 {\rm dom}(\mathbb{K}_{\Phi,\mu}):=\{g\in L^2(\mathbb{R}): \mathbb{K}_{\Phi,\mu}g\in  L^2(\mathbb{R})\}
 $$
 is closed (see, e.~g.,  \cite[p. 63]{Halmos_S}).
 \end{remark}

\begin{corollary} The operator $\mathbb{H}_{\Phi,\mu}$ is bounded in $PW$ if and only if the operator  $\mathbb{K}_{\Phi,\mu}$ in the right-hand side of \eqref{K} is bounded in $PW$. In this case \eqref{K} holds for all $f\in PW$.
\end{corollary}

\begin{corollary} Let $K_{\Phi,\mu}\in L^2(\mathbb{R}^2)$. Then $\mathbb{H}_{\Phi,\mu}$ is a Hilbert-Schmidt operator in $PW$ if and only if it is bounded in $PW$.  
\end{corollary}

\begin{proof} Let  $\mathbb{H}_{\Phi,\mu}$  is bounded in $PW$.   Since the operator $\mathbb{K}_{\Phi,\mu}$ in the right-hand side of \eqref{K} is  Hilbert-Schmidt  in $L^2(\mathbb{R})$, it is bounded in $L^2(\mathbb{R})$. Therefore \eqref{K} holds for all $f\in PW$. Since $PW$ is a closed invariant subspace of the operator $\mathbb{K}_{\Phi,\mu}$ (the restriction of $\mathbb{K}_{\Phi,\mu}$ to $PW$ equals to  $\mathbb{H}_{\Phi,\mu}$),  the  restriction of $\mathbb{K}_{\Phi,\mu}$ to $PW$ is  Hilbert-Schmidt, as well, since the Hilbert-Schmidt property is hereditary.

 The ``only if'' statement is obvious.

\end{proof}

\begin{corollary} Let the operator $\mathbb{K}_{\Phi,\mu}$ in the right-hand side of \eqref{K} is nuclear in $L^2(\mathbb{R})$. Then $\mathbb{H}_{\Phi,\mu}$ is nuclear in $PW$  if and only if it is bounded in $PW$.
\end{corollary}

\begin{proof}  The proof is similar to the proof of the previous corollary.
\end{proof}

Analogously we have the following corollary.

\begin{corollary} Let the operator  $\mathbb{K}_{\Phi,\mu}$ in the right-hand side of \eqref{K} is compact  in $L^2(\mathbb{R})$. Then $\mathbb{H}_{\Phi,\mu}$ is  compact in $PW$ if and only if it is bounded in $PW$. In this case,  \eqref{K} holds for all $f\in PW$.
\end{corollary}

\begin{corollary}  The operator $\mathbb{H}_{\Phi,\mu}$ is a bounded as an operator between  $PW$ and $L^\infty(\mathbb{R})$ if (and only if)  $\Phi\in L^1(\mu)$.  In this case,  $\|\mathbb{H}_{\Phi,\mu}\|_{PW\to L^\infty}\le \|\Phi\|_{L^1}$.
\end{corollary}

\begin{proof}  Let $\Phi\in L^1(\mu)$. If $\mathbb{K}_{\Phi,\mu}$ denotes the integral operator in the right-hand side of \eqref{K}, one has by the Caychi-Bunyakovski's inequality and \eqref{Carleman} that for all $t\in \mathbb{R}$
$$
|\mathbb{K}_{\Phi,\mu}f(t)|\le  \int_0^\infty |\Phi(u)|\,d\mu(u)\|f\|_{L^2( \mathbb{R})}.
$$
Therefor the operator $\mathbb{K}_{\Phi,\mu}$ 
 is a bounded as an operator between  $L^2( \mathbb{R})$ and $L^\infty(\mathbb{R})$, and $\|\mathbb{K}_{\Phi,\mu}\|_{L^2\to L^\infty}\le \|\Phi\|_{L^1}$.  Since $D$ is dense in $PW$, formula \eqref{K} yields that  $\mathbb{H}_{\Phi,\mu}$ is a bounded as an operator between  $PW$ and $L^\infty(\mathbb{R})$ and $\|\mathbb{H}_{\Phi,\mu}\|_{PW\to L^\infty}\le \|\Phi\|_{L^1}$.
 
 The ``only if'' statement is obvious (see Proposition \ref{th:PW}). 

\end{proof}

\begin{corollary}
Assume  that $\Phi\in L^1(\mu)$ and $\Phi(u)=0$ $\mu$-a.e. on the interval $(0, 1)$. If  the operator  $\mathbb{K}_{\Phi,\mu}$ in the right-hand side of \eqref{K} is bounded in $L^2(\mathbb{R})$ then  $\mathbb{H}_{\Phi,\mu}$ is bounded in $PW$.
\end{corollary}

\begin{proof} The function  $\mathbb{H}_{\Phi,\mu}f$ is entire for $f\in PW$ (see the proof of Theorem \ref{th:31}). Since  $\mathbb{K}_{\Phi,\mu}$  is bounded in $L^2(\mathbb{R})$, the  equality  \eqref{K} shows that the restriction  of  $\mathbb{H}_{\Phi,\mu}f$ to  $\mathbb{R}$  belongs to $L^2(\mathbb{R})$  for $f\in D$.  Next,  if $|f(z)|\le Me^{\pi |z|}$  for sufficiently large $z$ then
\begin{align*}
|\mathbb{H}_{\Phi,\mu}f(z)|&\le M\int_1^\infty |\Phi(u)|e^{\pi \frac{|z|}{u}}\,d\mu(u)\\
&\le  M\int_1^\infty |\Phi(u)|\,d\mu(u)e^{\pi |z|},
\end{align*}
 for such $z$, and so, $\mathbb{H}_{\Phi,\mu}f\in PW$. Finally, for  $f\in D$ we have by \eqref{K}
\begin{align*}
\|\mathbb{H}_{\Phi,\mu}f\|_{PW}&=\|\mathbb{H}_{\Phi,\mu}f\|_{L^2}=\|\mathbb{K}_{\Phi,\mu}f\|_{L^2}\\
&\le \|\mathbb{K}_{\Phi,\mu}\|\|f\|_{L^2}=\|\mathbb{K}_{\Phi,\mu}\|\|f\|_{PW}.
\end{align*}
Since $D$ is dense in $PW$, this completes the proof.
\end{proof}

The following two corollaries give us an information about extensions of a Hausdorff operator.

\begin{corollary}\label{cor:semi-Carleman}
Assume  that $\Phi(u)\sqrt{u}\in L^1(\mu)$. Then the operator $\mathbb{K}_{\Phi,\mu}$ is semi-Carleman in $L^2(\mathbb{R})$ (i.~e. $\forall x K_{\Phi,\mu}(\cdot, x)  \in L^2(\mathbb{R})$, see \cite{ST}). Moreover, $\mathbb{K}_{\Phi,\mu}$ is bounded as an operator between  $L^1(\mathbb{R})$  and  $L^2(\mathbb{R})$ and its norm does not exceed $\int_0^\infty |\Phi(u)|\sqrt{u}d\mu(u)$.
\end{corollary}

\begin{proof} Similar to the proof of formula \eqref{Carleman} one can show that for all $x\in \mathbb{R}$
$$
 \|K_{\Phi,\mu}(\cdot, x)\|_{L^2}\le \int_0^\infty |\Phi(u)|\sqrt{u}d\mu(u).
$$.
Further, by the Minkowski inequality we have for $L^1(\mathbb{R})$
\begin{align*}
\|K_{\Phi,\mu}f\|_{L^2}&\le \int_{-\infty}^\infty\|K_{\Phi,\mu}(\cdot, x)\|_{L^2}|f(x)|\,dx\\
&\le \int_0^\infty |\Phi(u)|\sqrt{u}d\mu(u)\|f\|_{L^1}.
\end{align*}
\end{proof}

If an operator $T$ is Carleman and semi-Carleman, it is called  bi-Carleman \cite{ST}.

\begin{corollary}
Assume  that $\Phi(u), \Phi(u)\sqrt{u}\in L^1(\mu)$, and $\mathbb{K}_{\Phi,\mu}$ is bounded in $L^2(\mathbb{R})$.  Then $\mathbb{K}_{\Phi,\mu}$ is bi-Carleman, and its adjoint  is
represented by the transposed kernel, and is therefore again a bi-Carleman
operator.
\end{corollary}

\begin{proof}
Since the operator $\mathbb{K}_{\Phi,\mu}$ is bi-Carleman by Theorem \ref{th:K} and Corollary \ref{cor:semi-Carleman} and bounded in  $L^2(\mathbb{R})$, the statement of the corollary follows from   \cite[Theorem 3.3]{ST}.
\end{proof}


\begin{thebibliography}{99}

\bibitem{Baranov}
A. Baranov, ``Polynomials in the de Branges spaces
of entire functions'', Ark. Mat., {\bf 44}, 16--38,  (2006) DOI: 10.1007/s11512-005-0006-1

\bibitem{BBB}
A. Baranov, Yu. Belov, A. Borichev, ``Spectral synthesis in de Branges spaces'', GAFA, {\bf 25}, 417--452  (2015).

\bibitem{BaranovBo}
A. Baranov,  H.  Bommier-Hato,  ``de Branges
spaces and Fock-type spaces'', Complex Var. Elliptic Equ.,  {\bf 63}, no. 7--8, 907--930,  (2018).


\bibitem{Blasco} O. Blasco and A. Galbis,    ``Boundedness and compactness of Hausdorff operators on Fock spaces'',
 Trans. Amer. Math. Soc. {\bf 377}, 5165--5196,  (2024), DOI: https://doi.org/10.1090/tran/9133,  arXiv:2310.01059v2.

\bibitem{Bonet}
J. Bonet,  ``Hausdorff operators on weighted Banach spaces of type $H^\infty$'',  Complex Anal. Oper. Theory, {\bf 16}, no.1, Paper No. 12, 14 pp  (2022).


\bibitem{dB}
L. de Branges, {\it Hilbert Spaces of Entire Functions} (Prentice Hall, Englewood Cliffs (NJ), 1968).


\bibitem{Dym McKean}
H. Dym, H. McKean, ``Application of de Branges spaces of integral functions to the
prediction of stationary Gaussian processes'', Illinois J. Math. {\bf 14}, 299--343  (1970).

\bibitem{KGM} 
S. Grudsky, A. Karapetyants, and A. R. Mirotin, ``Estimates for singular numbers of Hausdorff-Zhu operators and applications'', Math. Meth. Appl., Sci.  {\bf 46}, 9676--9693  (2023).

\bibitem{Halmos_S} 
P.R. Halmos, V.S. Sunder, {\it Bounded Integral Operators on $L^2$ Spaces} (Springer-Verlag, Berlin-Heidelberg-New York, 1978). 

 


\bibitem{Higgins}
J. R. Higgins, ``Five short stories about the cardinal series'', Bulletin (New Series) of Amer. Math. Soc., { \bf 12}, no 1, 45--89  (1985).




\bibitem{HKQ} H. D. Hung, L. D. Ky, and T. T. Quang,  ``Hausdorff operators on holomorphic Hardy spaces and applications'', Proc. Roy. Soc. Edinburgh Sect. A, {\bf 150}, 1095--1112  (2020).


\bibitem{KM} A.  Karapetyants and  A. R. Mirotin,  ``A class of Hausdorff-Zhu operators'', Analysis and Mathematical Physics,  {\bf 13},  1--19  (2022).


\bibitem{KS}
N. Karapetiants, S. Samko,  {\it Equations with Involutive Operators} (Birkhaser Boston Inc, Boston, 2001).

\bibitem{Korotkov}
V. B. Korotkov, ``On integral operators with Carleman kernels'', Dokl. Akad.
Nauk SSSR, {\bf 165}, 748--751  (1965),  Soviet Math. Dokl., {\bf 6}, 1496--1499  (1965).

\bibitem{Mattner}
 L. Mattner, ``Complex differentiation under the integral'',  Nieuw Arch. Wiskd. {\bf 5/2}, no 2, 32--35 (2001).

\bibitem{LM:survay}
 E. Liflyand, A. R. Mirotin, ``Hausdorff Operators: Problems and Solutions'',
J. Math. Sci. A,  {\bf 280}, 1065--1114 (2024) https://doi.org/10.1007/s10958-024-07383-8.

\bibitem{MCAOT}   A. R. Mirotin,    ``Hausdorff Operators on Some Spaces of Holomorphic Functions on the Unit Disc'', Complex Anal. Oper. Theory, {\bf 15}, Paper No. 85, 16 pp  (2021),  https://doi.org/10.1007/s11785-021-01128-0

\bibitem{MCAOT2}  A. R. Mirotin,    ``Generalized Hausdorff-Zhu Operators on M\"{o}bius Invariant Spaces'', Complex Anal. Oper. Theory,  {\bf 16}, Paper No. 100 (2022).  https://doi.org/10.1007/s11785-022-01278-9

\bibitem{arxGeneral}
A. R. Mirotin, ``On a general concept of a Hausdorff-type operator'', arXiv:2308.02388 (2024).


\bibitem{ST}
M. Schreiber, Gy. Targonski, ``Carleman and Semi-Carleman Operators'', Proc. Amer. Math. Soc.,{\bf 24}, no. 2 (1970),  293--299.

\bibitem{S} G. Stylogiannis,  ``Hausdorff operators on Bergman spaces of the upper half plane'', Concrete Operators, {\bf 7}, no 1 (2020), 69--80.

\bibitem{SG} G. Stylogiannis and P. Galanopoulos, ``Hausdorff operators on Fock spaces and a coefficient multiplier problem'', Proc. Amer. Math. Soc. {\bf 151} (2023), 3023--3035
DOI: https://doi.org/10.1090/proc/16374


\bibitem{Weidman}
J. Weidmann, {\it Linear Operators in Hilbert Spaces} (Springer-Verlag, New York - Heidelberg - Berlin, 1980). 
 

\end{thebibliography}
\end{document}